\documentclass[a4paper, 11pt]{article}

\usepackage[utf8]{inputenc}
\usepackage[T1]{fontenc}
\usepackage[a4paper, margin=2.5cm]{geometry}
\usepackage{needspace}
\usepackage{nicefrac}

\usepackage{libertine} 
\usepackage{inconsolata} 
\usepackage{mathrsfs}
\usepackage{amsmath, amsthm, amssymb}
\usepackage{graphicx}
\usepackage{comment}
\usepackage{enumerate}
\usepackage{enumitem}
\usepackage{caption}

\usepackage{thmtools}
\usepackage{thm-restate}
\usepackage[hidelinks]{hyperref}
\usepackage{float}
\usepackage[normalem]{ulem}
\usepackage[capitalise]{cleveref}
\usepackage{mathtools}
\usepackage{nicematrix}

\newcommand{\eps}{\varepsilon}
\newcommand{\set}[1]{\{#1\}}

\declaretheorem[name=Theorem, numberwithin=section]{theorem}
\declaretheorem[name=Example, numberwithin=section]{example}
\declaretheorem[name=Lemma, sibling=theorem]{lemma}

\declaretheorem[name=Definition, sibling=theorem]{definition}

\declaretheorem[name=Claim, sibling=theorem]{claim}
\declaretheorem[
  name=Bound,
  refname={bound,bounds},
  Refname={Bound,Bounds}
]{bound}

\declaretheorem[name=Observation, style=remark, sibling=theorem]{observation}

\def\cqedsymbol{\ifmmode$\lrcorner$\else{\unskip\nobreak\hfil
\penalty50\hskip1em\null\nobreak\hfil$\lrcorner$
\parfillskip=0pt\finalhyphendemerits=0\endgraf}\fi}

\newenvironment{claimproof}[1][Proof.]{\begin{proof}[#1]}{\end{proof}}

\crefname{claim}{Claim}{Claims}

\usepackage[style=alphabetic,maxnames=999]{biblatex}
 \AtEveryBibitem{
 	\clearfield{editor}
 	\clearlist{editor}
 	\clearname{editor}
 	\clearfield{location}
 	\clearlist{location}
 	\clearname{location}
 	\clearfield{isbn}
 	\clearlist{isbn}
 	\clearname{isbn}
 	\clearfield{url}
 	\clearlist{url}
 	\clearname{url}
 	\clearfield{issn}
 	\clearlist{issn}
 	\clearname{issn}
 }

\def\C{\mathscr{C}} 
\def\CC{\C} %
\def\N{\mathbb{N}} 
\newcommand{\NN}{\mathbb{N}} 
\def\O{\mathcal O} 
\def\P{\mathcal{P}}
\def\Q{\mathcal{Q}}

\newcommand{\dist}{\operatorname{dist}}

\def\tw{\operatorname{tree-width}}

\def\mw{\operatorname{mw}}
\def\fw{\operatorname{fw}}
\def\degeneracy{\operatorname{degeneracy}}
\def\copw{\operatorname{copwidth}}
\def\smw{\operatorname{sw}}
\def\scol{\operatorname{scol}}
\def\wcol{\operatorname{wcol}}
\def\wreach{\operatorname{wreach}}
\def\sreach{\operatorname{sreach}}
\def\sep{\operatorname{sep}}
\newcommand{\Sparsify}{\mathsf{Sparsify}}
\def\frk{\mathrm{frk}}
\def\srk{\mathrm{srk}}

\let\le\leqslant
\let\ge\geqslant
\let\leq\leqslant
\let\geq\geqslant

\usepackage[textsize=small, textwidth=2cm, color=yellow]{todonotes}

\title{Flips and Merge-Width in Sparse Graphs%
\thanks{%
\begin{minipage}[b]{0.30\textwidth}
\includegraphics[scale=0.15]{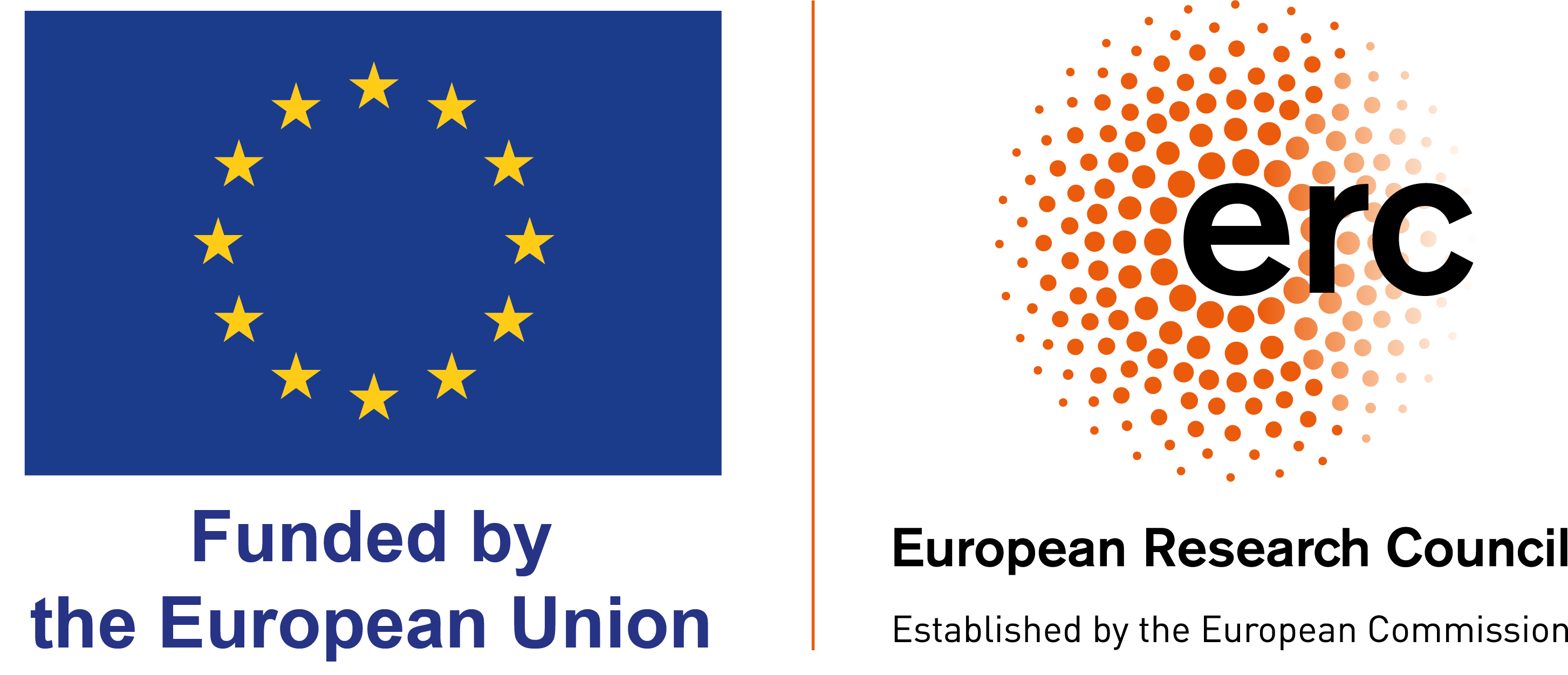}
\end{minipage}\hfill
\begin{minipage}[b]{0.66\textwidth}\footnotesize
KD, MD, NM, ST received funding from ERC grant BUKA (No.~101126229).
MD and WP were supported by the project BOBR that has received funding from ERC
under the European Union’s Horizon 2020 research and innovation programme,
grant agreement No.~948057.
\end{minipage}
}}

\date{}

\author{Karolina Drabik, Ma{\"e}l Dumas, Nikolas M{\"a}hlmann,\\Wojciech Przybyszewski, Szymon Toruńczyk
\\[9pt]
Institute of Informatics\\
University of Warsaw, Poland
}

\begin{document}

\maketitle

\begin{abstract}
    
A \emph{flip} of a graph is obtained by complementing the edge relation within a set of vertices. 
Flips are typically used to separate vertices in a graph, by increasing the distances between them.
We show that in $K_{t,t}$-free graphs, every short sequence of flips can be simulated by a short sequence of vertex deletions that achieves a similar degree of separation: distances in the resulting graph are, up to a factor of three, at least as large as those obtained after the flips.

This result provides a simple and uniform explanation of an emerging pattern in structural graph theory and finite model theory: the $K_{t,t}$-free fragment of a tameness notion for dense graphs often coincides with a tameness notion for sparse graphs.
As immediate applications, we recover the following known equivalences.
In the $K_{t,t}$-free setting, the dense notions (1)~bounded shrub-depth, (2)~bounded clique-width, (3) bounded flip-width, (4)~monadic dependence, respectively, coincide with the sparse notions (1)~bounded tree-depth, (2)~bounded tree-width, (3)~bounded expansion, and (4)~no\-where dense\-ness.

Furthermore, we reprove the result by Dreier and Toruńczyk (STOC 2025) stating that $K_{t,t}$-free classes of bounded merge-width have bounded expansion.
Our proof provides explicit bounds and is direct, as it shows how to construct strong coloring orders (witnesses of bounded expansion) from merge sequences (witnesses of bounded merge-width).

Along the way, we identify 
a new family of graph parameters, dubbed \emph{separation-width}, that is sandwiched between the strong and weak coloring numbers, and is closely related to the merge-width parameters.
We provide evidence that this family of graph parameters, apparently overlooked in the literature, may play a fundamental role in the study of sparse graphs.
\end{abstract}

\section{Introduction}

Sparsity theory, initiated by Ne{\v s}et{\v r}il and Ossona de Mendez~\cite{nevsetvril2012sparsity}, 
is a prolific framework providing various notions of structure for sparse graph classes.
Graph classes are categorized by whether they have certain hierarchical properties such as \emph{bounded expansion} and \emph{nowhere denseness}.
Any graph class which enjoys one of these properties also has various desirable combinatorial, algorithmic, and logical properties.

A more recent line of research in structural graph theory and finite model theory aims at 
lifting sparsity theory to the setting of graph classes which are possibly dense.
Here, some of the primary goals are to understand the behavior of first-order logic on graphs and to characterize the hereditary graph classes on which the first-order model checking problem is fixed-parameter tractable.
These problems are typically approached by finding ``dense extensions'' of the class properties studied in sparsity theory. 

As an example predating sparsity theory, the clique-width or rank-width parameters are extensions of tree-width to the setting of graphs which might contain arbitrarily large cliques.
Similarly, shrub-depth is a dense extension of tree-depth, bounded flip-width and bounded merge-width are dense extensions of bounded expansion, 
and monadic dependence and monadic stability are dense extensions of nowhere denseness.
See \Cref{fig:universe} for a depiction of the hierarchy containing most of the class properties mentioned in this paper, and \Cref{thm:sparse-classes} for specific references. We refer the reader to a survey of Pilipczuk~\cite{pilipczuk_survey} for a more in-depth discussion on these class properties and their relationships.

\begin{figure}[htbp]
    \centering
    \includegraphics[scale = 0.9]{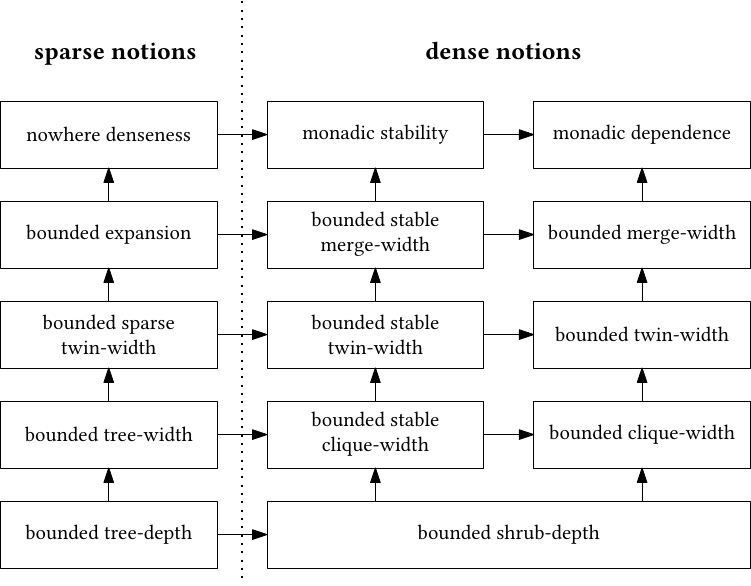}
    \caption{A hierarchy of most of the class properties mentioned in this paper (and some more). 
    Arrows represent implications between properties.
    In particular, \emph{bounded tree-depth} is the most restricted class property in this hierarchy and \emph{monadic dependence} is the most general one. All properties in the same row are equivalent when restricted to weakly sparse graph classes. 
    }
    \label{fig:universe}
\end{figure}

There is no formal, general definition of what it means that one notion is a ``dense extension'' of another notion (an attempt -- in special cases -- is made in \cite{stable-tww}).
A recurring theme,
however, is that the two notions coincide 
when the dense one is restricted to weakly sparse graph classes.
Here, a graph class is \emph{weakly sparse} if there is some $t\in\N$ such that all graphs in the class are $K_{t,t}$-free, that is, do not contain $K_{t,t}$ as a subgraph.
Moreover, in the examples above, the sparse notion coincides exactly with its dense counterpart, restricted to weakly sparse classes, as expressed informally below:
\begin{equation}\label{eq:analog}
\textit{sparse notion}\ =\ \textit{dense notion}\ \cap\, \textit{weakly sparse}\tag{$\star$}  
\end{equation}

For example, classes of bounded tree-width are exactly the same as weakly sparse classes of bounded clique-width \cite{10.1006/inco.1995.1020,courcelle2000upper,gurski2000}.
 Furthermore, each dense notion in \cref{fig:universe}, together with the sparse notion in the same row, forms an instance of \eqref{eq:analog}.
 See \Cref{thm:sparse-classes}, \cite[Sec. 7]{stable-tww} and 
 \cite[Sec. 21]{lmcs:9981} for more examples.

The systematic nature of equivalences of the form \eqref{eq:analog} suggests the presence of an underlying principle.
Nevertheless, we are not aware of an explanation that would subsume these equivalences in a unified framework.
At present, the equivalence between sparse notions and their dense extensions on weakly sparse classes remains a collection of isolated observations, proved individually for each notion.

\subsection*{Contribution 1: Sparsifying flips}
As our first contribution, we prove a simple lemma that explains many instances of the equation \eqref{eq:analog} and helps understand the phenomenon in a unified way.

An emerging pattern 
is that notions concerning sparse graphs, which are defined in terms of vertex deletions or vertex isolation,
 are typically generalized to the dense setting by replacing vertex deletions with flips.
A \emph{flip} of a graph $G$ is a graph $G'$ obtained from $G$ by picking a vertex set $X\subseteq V(G)$ and complementing the adjacency of vertex pairs in $X\choose 2$.
See \Cref{fig:flips} for an example.
Note that isolating a vertex $v$ in a graph (that is, removing precisely all edges incident to $v$) can be achieved by  applying two flips to $G$ (in any order):
\begin{itemize}
    \item flipping the open neighborhood $N_G(v)$ of $v$ in $G$, and 
    \item flipping the closed neighborhood $N_G[v]$ of $v$ in $G$.
\end{itemize}  
Similarly, isolating $k$ vertices can be achieved by performing a sequence of $2k$ flips.

\begin{figure}[htbp]
    \centering
    \includegraphics[scale = 1.1]{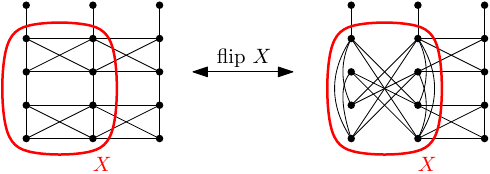}
    \caption{An example illustrating the flip operation.}
    \label{fig:flips}
\end{figure}

There are many recent examples that match the pattern expressed in~\eqref{eq:analog}, and where the sparse notion involving vertex isolations is generalized to a dense notion involving flips 
\cite{Tor23,flip-flatness, flipper-game, flip-breakability,flip-separability, ghasemi2025weaklysparsestronglyflipflatclasses}.
We present two of them here.

\begin{example}[Cops and Robber games]\label{ex:cops-and-robber}
In the Cops and Robber game characterizing tree-width \cite{seymour-thomas-cops}, 
we may think of the $k$ Cops as temporarily isolating $k$ vertices of the graph in each round. In the dense variant of this game \cite{Tor23},
instead of isolating vertices, 
the Cops temporarily apply a sequence of $k$ flips to the considered graph in each round. The least number $k$ of cops needed to win the game on a graph $G$ is functionally equivalent to the clique- and rank-width of~$G$.
\end{example}

\begin{example}[Flip-flatness]\label{ex:flip-flatness}
    Nowhere dense graph classes are characterized as exactly those which are \emph{uniformly quasi-wide}.
Very roughly, this notion means that for every fixed $r\in\N$, given a graph from the class and a large set of its vertices, we can find a relatively large subset of it which becomes $r$-scattered (the pairwise distances are greater than $r$) after a bounded number of vertices of $G$ are deleted (or isolated).
An analogous condition, called \emph{flip-flatness}, characterizes 
\emph{monadically stable} classes \cite{flip-flatness},
where instead of deleting a bounded number of vertices,
we apply a bounded number of flips to $G$.
\end{example}

Implicit in \Cref{ex:cops-and-robber} and explicit in \Cref{ex:flip-flatness}, the performed flips and vertex isolations/deletions serve the purpose of separating sets of vertices by increasing distances in the graph.
We prove a simple lemma, \Cref{lem:sparse-flips}, which allows us to explain many instances of equation \eqref{eq:analog} in a unified way. Intuitively, in the setting of sparse graphs, it allows one to replace flips with vertex deletions, without decreasing the distances too much.

\begin{lemma}[\Cref{thm:engine}, informally]\label{lem:sparse-flips}
Let $G$ be a $K_{t,t}$-free graph
and $G'$ be obtained by applying a sequence of $k$ flips to $G$. 
There is a set $S\subseteq V(G)$ whose size depends only on $k$ and $t$
such that
$$\dist_{G\setminus S}(u,v)\ge \tfrac{1}{3}\dist_{G'}(u,v)\qquad\text{for all $u,v\in V(G)\setminus S$.}$$
\end{lemma}

In the above, $G\setminus S$ denotes the graph $G$ with the vertices from $S$ removed.
Applied to the notions in \Cref{ex:flip-flatness}, \Cref{lem:sparse-flips} immediately yields that in weakly sparse graph classes, uniform quasi-wideness (characterizing nowhere denseness) and flip-flatness (characterizing monadic stability) coincide.
Similarly, we may readily derive the following known instances of equation \eqref{eq:analog} in a unified way,
with immediate proofs which amount to applying \Cref{lem:sparse-flips} to known characterizations of the appropriate dense notions in terms of flips.

\begin{theorem}\label{thm:sparse-classes}
For every weakly sparse class $\CC$, the following holds:
\begin{enumerate}
%

    \item\label{it:sd} If $\CC$ has bounded shrub-depth,
    then $\CC$ has bounded tree-depth. 
    \hfill\cite[Lem.\ 2.12]{sbe}
    \item\label{it:cw} If $\CC$ has bounded clique-width,
    then $\CC$ has bounded tree-width. 
    \hfill\cite[Thm.\ 5.9]{courcelle2000upper}
    \item\label{it:mon-stab} If $\CC$ is monadically stable,
    then $\CC$ is nowhere dense. 
    \hfill\cite[Thm.\ 6]{dvorak2018induced}
    \item\label{it:mon-dep} If $\CC$ is monadically dependent, 
    then $\CC$ is nowhere dense. 
    \hfill\cite[Thm.\ 6]{dvorak2018induced}
    \item    \label{it:fw} If $\CC$ has bounded flip-width,
    then $\CC$  has bounded expansion. 
    \hfill\cite[Thm. 6.3]{Tor23}
    \item\label{it:stronglyflip} If $\CC$ is strongly flip-flat,
    then $\CC$ is uniformly almost-wide. 
    \hfill\cite[Main Thm.]{ghasemi2025weaklysparsestronglyflipflatclasses}
\end{enumerate}
\end{theorem}
See \Cref{sec:flatness} for the definitions of the notions considered in \Cref{thm:sparse-classes}, for its proof, and  
for more examples of this principle.
In the same way, we may also derive results which have not yet been conceived in the literature. 

\paragraph*{Further work related to Contribution 1}

    Bonnet et al. \cite[Thm. 3.5]{incremental} (see also \cite[Lem. 9.8]{Tor23}), and Mählmann and Siebertz~\cite{MahlmannSiebertz26E}
    state lemmas similar in spirit to \Cref{lem:sparse-flips}. They allow one to convert flips to \emph{definable flips} in graphs of bounded VC-dimension and to \emph{subflips} in co-matching-free graphs, respectively. In \Cref{lem:sparse-flips}, the assumption is stronger (weakly sparse classes have bounded VC-dimension and are co-matching-free) and the conclusion is stronger (isolating vertices in $S$ is a special case of both an $S$-definable flip and a subflip).

\subsection*{Contribution 2: Merge-width in sparse graphs}

Our second contribution is a detailed study of the following instance of equation~\eqref{eq:analog}.

\begin{restatable}[{\cite[Cor. 1.8]{DT25}}]{theorem}{thmMW}
\label{thm:mw-sparsify}
    A graph class has bounded expansion if and only if it is weakly sparse and has bounded merge-width.
\end{restatable} 

The \emph{radius-$r$ merge-width} graph parameters (denoted $\mw_r$) were recently introduced by Dreier and Toruńczyk with the aim of generalizing results concerning classes of bounded expansion to the dense setting~\cite{DT25}.
A graph class $\CC$ has bounded merge-width if $\mw_r(\CC) < \infty$ for every $r\in\N$.

The original proof of \Cref{thm:mw-sparsify} given in~\cite{DT25} uses the fact that classes of bounded merge-width have bounded flip-width, and then concludes by applying \Cref{thm:sparse-classes} \eqref{it:fw} above. It does not explicitly construct a witness for having bounded expansion -- instead it proves the inexistence of dense \emph{shallow minors}, an obstruction to having bounded expansion~\cite{NOdM08}.
In this paper, we give a direct proof of \Cref{thm:mw-sparsify} that yields explicit bounds and is constructive in the following sense: given a $K_{t,t}$-free graph and a witness that $G$ has small merge-width, we explicitly construct an ordering on the vertex set of $G$ which witnesses that $G$ has small \emph{strong $r$-coloring numbers} (denoted $\scol_r$). In particular, our proof shows the following.

\begin{lemma}\label{lem:bounded_scol}
    Fix $d,t,r\in\N$ with $r\ge 1$.
    Let $G$ be a $K_{t,t}$-free graph with $\mw_{3r+1}(G) \leq d$.  Then 
    \[\scol_r(G) \leq \O(d^2t^3).\]
\end{lemma}

As a graph class $\CC$ has bounded expansion if $\scol_r(\CC) < \infty$ for every $r\in\N$~\cite{Zhu09}, this reproves the harder implication in \Cref{thm:mw-sparsify}.
The previously discussed instances of \eqref{eq:analog} were immediate applications of \Cref{lem:sparse-flips}. Our proof of \Cref{lem:bounded_scol} uses \Cref{lem:sparse-flips} in a more sophisticated way and is the main technical contribution of this paper. 

Specifically, if a graph $G$ has radius-$r$ merge-width $k$,
then this is witnessed by a refining sequence of partitions 
$\P_1\succ \P_2\succ \ldots\succ\P_n$, where $n=|V(G)|$ and $\P_i$ has $i$ parts for $i=1,\ldots,n$, 
as well as a sequence of graphs $G_1,G_2,\ldots,G_n$,
where $G_i$ is a \emph{$\P_i$-flip of $G$} -- it is obtained from $G$ by possibly inverting the adjacency between some pairs $A,B$ of parts of $\P_i$.
Moreover, every radius $r$-ball in~$G_i$ intersects at most $k$ parts of the partition $\P_i$. 
The idea behind our proof of \Cref{lem:bounded_scol}
is to apply \cref{lem:sparse-flips} to each of the graphs $G_i$,
obtaining a sequence of sets $S_1\subseteq\ldots\subseteq S_n\subseteq V(G)$,
such that for all $r'\in\N$,
\begin{equation}\label{eq:bb}
 B^{r'}_{G\setminus S_i}(v)\subseteq B^{3r'}_{G_i}(v)\qquad\text{for all $i\in[n]$ and $v\in V(G)\setminus S_i$.}   
\end{equation}
Now, we pick any total order on $V(G)$ which first places 
all vertices of $S_1$ (in any order), followed by all vertices of $S_2\setminus S_1$, etc. We then show, using \eqref{eq:bb},
that this total order 
witnesses that $G$ has a small strong coloring number for radius $r'=\lfloor (r-1)/3\rfloor$.

We believe that this proof, besides providing the explicit bounds stated in \cref{lem:bounded_scol} and being more constructive than the previous proof of \cref{thm:mw-sparsify}, 
 sheds new light on the relationship between merge-width and the fundamental strong coloring numbers studied in sparsity theory.

\paragraph*{Further work related to Contribution 2}
The approach of transforming 
merge sequences in a step-by-step fashion is explored further in \cite{def-mw}.

Let us also mention two previous results that provide direct and explicit bounds for known instances of \eqref{eq:analog}, in the same spirit as our \Cref{lem:bounded_scol}. 
Gurski and Wanke~\cite{gurski2000} give a direct proof showing that for $K_{t,t}$-free graphs, tree-width is bounded in terms of clique-width. This reproves a result of Courcelle and Olariu \cite{courcelle2000upper} by providing an explicit bound.

Bonnet, Geniet, Kim, Thomass\'e, and Watrigant \cite{tww2} prove that $K_{t,t}$-free classes of bounded \emph{twin-width} have bounded expansion, although without providing explicit bounds.
Dreier et al.~\cite{DREIER2022112746} give a direct proof of that fact, by providing an explicit bound on the (weak and strong) coloring numbers of such graphs.

\subsection*{Contribution 3: Separation-width}

As a last, conceptual contribution, 
we define a new family of graph parameters tailored to sparse graph classes, dubbed \emph{separation-width}. We argue that this notion occupies a sweet spot between the previously studied \emph{strong} and \emph{weak coloring numbers}, combining the small bounds characteristic of the former with much of the proof flexibility of the latter. 
The definition is as follows.

Fix $r\in \N\cup\set{\infty}$. For a graph $G$, a set $S$
of its vertices, and vertex $v\in V(G)\setminus S$, 
 the \emph{$r$-separator of $v$ in $S$}, denoted by $\sep_r(v/S)$, is the set of vertices in $S$ that can be reached from $v$ by a path of length at most~$r$, whose inner vertices are disjoint from $S$
 (this is called the \emph{$r$-projection of $v$ onto $S$} in  \cite{kernelization}).
    The  \emph{radius-$r$ separation-width} of a graph $G$, denoted by $\smw_r(G)$,
is the least number $k$ such that there is an enumeration $v_1,\ldots,v_n$ of $V(G)$ such that $\left|\sep_r(v_j/\{v_1,\ldots,v_i\})\right|\le k$ for all $i<j\le n$.

We collect several results, some observed here and some drawn from the literature, which support viewing separation-width as a fundamental notion. 
Here, $\scol_r$ and $\wcol_r$ denote the strong and weak $r$-coloring numbers, respectively (see \Cref{sec:sw}).

\begin{restatable}{lemma}{lemScolSmw}\label{lem:scol_smw}
    For every graph $G$ and $r\in\N\cup\set{\infty}$, $r\ge 1$, we 
     have 
     \begin{align}
     \scol_r(G)&\le \smw_r(G)+1\le \wcol_r(G),\label{eq:sandwich1}\\
     \scol_r(G)&\le \smw_r(G)+1\le \scol_{2r-1}(G).\label{eq:sandwich2}
     \end{align}
\end{restatable}
The following statements are immediate corollaries, since these properties are already enjoyed by the strong coloring numbers (minus one)~\cite{generalizedcoloring,DBLP:conf/wg/GroheKRSS15,Zhu09}. 
\begin{restatable}{corollary}{corSMWDEGTW}\label{cor:smw-tw}
    For every graph $G$, we have
    \begin{equation}
        \smw_1(G)=\degeneracy(G)\quad\text{and}\quad
    \smw_\infty(G)=\tw(G).\label{eq:smw-tw}
    \end{equation}
\end{restatable}

\begin{restatable}{corollary}{corSMWBE}\label{cor:smw-be}
    A graph class $\CC$ has bounded expansion if and only if $\smw_r(\CC)<\infty$ for every $r\in\N$. 
\end{restatable}

We argue that several results in the literature, which provide upper bounds in terms of weak coloring numbers, can be strengthened and made more direct by using separation-width instead (specifically, upper bounds on merge-width and copwidth, see \Cref{sec:sw}).
A posteriori, by \eqref{eq:sandwich2}, this also provides an upper bound in terms of strong coloring numbers.
Strong coloring numbers may be much smaller than weak coloring numbers\footnote{For instance, for every $r,t\in \N$ there is a graph $G$ of tree-width $t$ with  $\wcol_r(G)= \binom{r+t}{t}=\Theta_r(t^r)$ \cite[Sec. 4]{DBLP:conf/wg/GroheKRSS15},
while by \Cref{cor:smw-tw}, $\smw_r(G)\le \smw_\infty(G)= t$.}, and this property is inherited by separation-width by \eqref{eq:sandwich2}; for this reason, upper bounds in terms of separation-width or strong coloring numbers are stronger than in terms of weak coloring numbers.

 We observe that separation-width directly yields 
 a winning strategy for the Cops in the Cops and Robber game, in which the robber moves at speed $r$. 
Moreover, this strategy is \emph{monotone}: the robber never visits a vertex which was previously occupied by a cop.
This is analogous to (and generalizes) the classic result that tree decompositions yield descriptions of monotone winning strategies in the classic Cops and Robber game (where the robber moves at infinite speed). 

 Furthermore, we observe that for $K_{t,t}$-free graphs,
 separation-width and merge-width are tightly related to each other,
 which supports the view that merge-width is a dense extension of separation-width.
%

\section{Preliminaries}

\paragraph*{Graphs.} In this paper, graphs are finite, undirected and without self-loops and parallel edges.
The set of vertices of a graph $G$ is denoted $V(G)$, and the set of edges of $G$ is denoted $E(G)$.  An unordered pair $\set{a,b}$ of distinct vertices is denoted $ab$. 
For $A,B\subseteq V(G)$ denote $AB\coloneqq \{ab\mid a\in A,b\in B, a\neq b\}$.
For  a vertex $v$ of a graph $G$  the (open) \emph{neighborhood} of $v$ in $G$ is $N_G(v)\coloneqq \{u \mid uv\in E(G)\}$,  
denoted $N(v)$ if $G$ is understood from the context.
The $r$-ball of $v$ in $G$, denoted $B_G^r(v)$, is the set of vertices at distance at most $r$ from $v$ in $G$.

A graph $H$ is a \emph{subgraph} of $G$ if $H$ is obtained by removing vertices and/or edges from $G$, 
and is an \emph{induced subgraph} of $G$ if $H$ is obtained by removing vertices from $G$, alongside with the edges incident to them.
The subgraph of $G$ induced by a set of vertices $X\subseteq V(G)$ is the graph $G[X]$ with vertices $X$ and edges $uv\in E(G)$ with $u,v\in X$.
A graph $G$ is \emph{$K_{t,t}$-free} if no subgraph of $G$ is isomorphic to the biclique $K_{t,t}$.
A graph class $\C$ is \emph{weakly sparse} if there is some $t\in\N$ such that every $G\in \C$ is $K_{t,t}$-free.

\paragraph*{Flips.} For a graph $G$ and 
two sets $A,B\subseteq V(G)$,
\emph{flipping} the pair $(A,B)$ in $G$ 
results in the graph $H$ on the same vertex set, where for every two distinct vertices $u, v \in V(G)$
\[
    uv \in E(H)
    \quad 
    \Leftrightarrow
    \quad \begin{cases}
        uv \notin E(G) & \text{if } uv \in AB \\
        uv \in E(G) & \text{otherwise.}
    \end{cases}
\] 

Given a partition $\P$ of $V(G)$, a \emph{$\P$-flip} of $G$ is a graph $H$
obtained from $G$ by performing a sequence of flips between pairs of parts $A,B\in\P$ (possibly with $A=B$).
A \emph{\(k\)-flip} is a \(\P\)-flip with \(\left|\P\right| \le k\).
The following properties of flips are readily seen.

\begin{lemma}\label{lem:flip-basics} For all graphs $G$, $H$, $H'$ on the same vertex set $V(G)$ and $S \subseteq V(G)$ we have:
    \begin{enumerate}
        \item \emph{Symmetry:} If $H$ is a $k$-flip of $G$, then $G$ is also a $k$-flip of $H$.
        \item \emph{Transitivity:}     If $H'$ is a $k'$-flip of $H$ and $H$ is a $k$-flip of $G$, then $H'$ is a $(k \cdot k')$-flip of $G$.
        \item \emph{Hereditariness:} If $H$ is a $k$-flip of $G$, then $H[S]$ is also a $k$-flip of $G[S]$.
    \end{enumerate}
 
\end{lemma}
\section{Flips in sparse graphs}
In this section, we give a simple proof of the following lemma.

\begin{restatable}[Flip sparsification lemma]{lemma}{thmEngine}
  \label{thm:engine}Fix $t\in\N$.
    Let $G$ be a $K_{t,t}$-free graph and~$\P$ be a partition of $V(G)$.
    There exists a set $S\subseteq V(G)$ of size at most $\left|\P\right|\cdot t^2$ such that for every $\P$-flip $H$ of $G$
    we have $$\dist_{H}(u,v) \leq 3\qquad \text{for all }uv\in E(G\setminus S).$$
    In particular,
    \[
        \dist_{G\setminus S}(u,v) \ge \tfrac{1}{3}\dist_{H}(u,v)\qquad\text{ for all  $u,v \in V(G)\setminus S$.}
    \]
\end{restatable}

To construct the separating set $S$ that approximates the flip, we need the following definitions.

\begin{definition}
    Fix a graph $G$, $t\in\N$, $P \subseteq V(G)$, and $v \in V(G)$.
    \begin{itemize}
        \item The set $P$ is \emph{$t$-small} if $\left|P\right| < t^2$.
        \item The set $P$ is \emph{$t$-big} if it is not $t$-small.
        \item The vertex $v$ is \emph{$t$-complete} to $P$ in $G$ if $\left|P\setminus N(v)\right| < t$.
    \end{itemize}
\end{definition}



\begin{definition}
    For a graph $G$, a partition $\P$ of $V(G)$, and $t \in \N$, we define $\Sparsify(G,\P,t)$ to be the set of all vertices $v\in V(G)$ that are either 
    \begin{enumerate}
        \item contained in a $t$-small part of $\P$, or
        \item $t$-complete to at least one $t$-big part of $\P$.      
    \end{enumerate}
\end{definition}

\begin{lemma}\label{lem:sparsify}
    For every $K_{t,t}$-free graph $G$, partition $\P$ of $V(G)$, and $\P$-flip $H$ of $G$, we have
    \[
        \dist_{H}(u,v) \leq 3\qquad \text{for all }uv\in E\big(G\setminus \Sparsify(G,\P,t)\big).
    \]
\end{lemma}
\begin{proof}
	Fix an edge $uv\in E\big(G\setminus \Sparsify(G,\P,t)\big)$ and let $P, Q\in \P$ be such that $u\in P$ and $v\in Q$ (possibly with $P=Q$).
	If $P$ and $Q$ are not flipped in $H$, then $\dist_{H}(u, v) = 1$, and we are done.
	Suppose now that $P$ and $Q$ are flipped, and define
    \[
        Q' \coloneq  Q \setminus N_G(u)
        \quad
        \text{and}
        \quad
        P' \coloneq  P \setminus N_G(v).
    \]
    See \cref{fig:sparsify_proof} for an illustration of these sets.

    \begin{claim}\label{clm:big-q-prime}
        $|Q'| \geq t$ and $|P'| \geq t$.
    \end{claim}
    \begin{claimproof}
        The part $Q$ was not deleted, hence it must be $t$-big.
        Moreover, $u$ was not deleted, so $u$ is not $t$-complete to $Q$. Then by definition of $t$-completeness, $|Q'| \geq t$. Analogously, also $|P'| \geq t$. 
    \end{claimproof}
    \begin{claim}\label{clm:dist-to-set}
        For all $x \in Q'$ and $y \in P'$, we have
        \[
        \dist_H(u,x) \leq 1
        \quad
        \text{and}
        \quad
        \dist_H(v,y) \leq 1.
        \]
    \end{claim}
    \begin{claimproof}
        If $u = x$, then $\dist_H(u,x) = 0$.
        Otherwise, $u$ and $x$ are distinct. Since $x \in Q \setminus N_G(u)$ and the adjacency between $u \in P$ and $Q$ was flipped in $H$, the vertices $u$ and $x$ are adjacent in $H$.
        Analogously, also $\dist_H(v,y) \leq 1$.
    \end{claimproof}
    If $Q'$ and $P'$ overlap, then $\dist_H(u,v) \leq 2$ by \Cref{clm:dist-to-set}, and we are done (see \cref{fig:sparsify_proof}).
    Therefore, assume $Q'$ and $P'$ are disjoint.

    Assume towards a contradiction that $Q'P' \cap E(H) = \emptyset$.
    Because $Q \supseteq Q'$ and $P\supseteq P'$ are flipped, and because $Q'$ and $P'$ are disjoint, we must have $Q'P' \subseteq E(G)$.
    By \Cref{clm:big-q-prime}, $Q'$ and $P'$ form a $K_{t,t}$ in $G$; a contradiction.

    Otherwise, there is an edge between $Q'$ and $P'$ in $H$, and we conclude that $\dist_H(u,v) \leq 3$ by \Cref{clm:dist-to-set}.
\end{proof}

    \begin{figure}[]
	\centering
	\includegraphics[scale = 1.4]{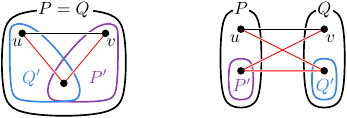}
	\caption{Illustration of the cases $P=Q$ and $P\neq Q$ in  \cref{lem:sparsify}. Edges in black are from the original graph $G$, edges in \textcolor{red}{red} are from the $\P$-flip $H$.}
	\label{fig:sparsify_proof}
\end{figure}

It remains to bound the size of $\Sparsify(G,\P,t)$, which we do with the next two lemmas.

\begin{lemma}\label{lem:almost_complete}
    Let $G$ be a $K_{t,t}$-free graph and let $P$ be a $t$-big subset of $V(G)$. Less than $t$ vertices of $G$ are $t$-complete to $P$.
\end{lemma}
\begin{proof}
Assume towards a contradiction that there is a set $A \subseteq V(G)$ of $t$ many vertices that are all $t$-complete to $P$.
Each vertex in $A$ has less than $t$ non-neighbors in $P$.
This means there must be a set $B \subseteq P$ of size $|B| \geq |P| - t(t-1) \geq t$ of vertices that are adjacent to every vertex in $A$. This implies that $A$ and $B$ are disjoint (each vertex is non-adjacent to itself), and form a $K_{t,t}$ in $G$; a contradiction.   
\end{proof}

\begin{lemma}\label{lem:sparsify-size}
    For every $K_{t,t}$-free graph $G$ and partition $\P$ of $V(G)$, we have
    \[
        \left|\Sparsify(G,\P,t)\right| < \left|\P\right|\cdot t^2.
    \]
\end{lemma}
\begin{proof}
	The set $\Sparsify(G,\P,t)$ contains every $t$-small part of $\P$, each of size less than $t^2$. Moreover for each $t$-big part, by \cref{lem:almost_complete}, $\Sparsify(G,\P,t)$ contains less than $t$ vertices.
\end{proof}

\Cref{thm:engine} follows immediately from \Cref{lem:sparsify-size} and \Cref{lem:sparsify}.

\begin{lemma}\label{lem:sparsify_refinement}
   For every $K_{t,t}$-free graph $G$, partition $\P$ of $V(G)$, and refinement $\P'$ of $\P$ with $\left|\P'\right| = \left|\P\right| +1$, we have
   \[
       \left| \Sparsify(G,\P',t) \setminus \Sparsify(G,\P,t) \right| < 2t^2.
   \]
\end{lemma}
\begin{proof}
	Let $P \in \P$ such that there are two distinct non-empty parts $P_1,P_2 \in \P'$ with $P = P_1 \cup P_2$ and let $D \coloneq  \Sparsify(G,\P',t) \setminus \Sparsify(G,\P,t)$. 
    For each part $Q \in \{ P_1,P_2 \}$, it is either
    \begin{itemize}
        \item $t$-small, and it can add to $D$ only the ${<}t^2$ vertices it contains, or
        \item $t$-big, and it can add to $D$ only the ${<}t$ vertices that are $t$-complete to it (by \Cref{lem:almost_complete}).\qedhere
    \end{itemize}
\end{proof}

\section{Immediate applications of \Cref{thm:engine}}\label{sec:flatness}

In this section we showcase a series of immediate applications of our flip sparsification lemma, which will amount to a proof of \Cref{thm:sparse-classes}.

\subsection{Shrub-depth, clique-width, monadic stability, and monadic dependence}

The left column of \Cref{table} lists various characterizations of the class properties bounded tree-depth, bounded tree-width, and nowhere denseness.
Each of these characterizations involves deleting few vertices with the purpose of separating sets of vertices in the graph.

For the more restricted notions of bounded tree-depth and -width, we require a strong notion of separation: after the deletion, the separated vertices should be in different connected components (i.e., at distance $\infty$).
For the more general notion of nowhere denseness, a weaker notion of separation is used:
after the deletion, the separated vertices should be at distance ${\geq}r$ for some fixed $r\in\N$.
Observing the similarities between the characterizations it is reasonable to view nowhere denseness as a ``local generalization'' of both bounded tree-depth and -width. 

A recent line of research has lifted the deletion-based characterization for sparse classes to flip-based characterizations for dense classes in the right column of \Cref{table}.
The general pattern here is that one obtains a characterization of a dense notion by exchanging vertex deletions with flips.
In the case of bounded tree-depth and -width, one obtains characterizations for bounded shrub-depth and clique-width, respectively.
For nowhere denseness, depending on whether its characterizations from the first or second row of \Cref{table} are generalized, one obtains characterizations for either monadic dependence or monadic stability.

Our \Cref{thm:engine} can be used to automatically reverse this lifting process for sparse classes. It provides an easy way to show that every $K_{t,t}$-free class satisfying a property from the dense column also satisfies the corresponding property from the sparse column in the same row.

\newcommand{\heading}[1]{\underline{\textbf{#1:}}\smallskip\\}

\newcommand{\treedepth}{
    $\CC$ has bounded tree-depth\\
    $\Leftrightarrow$ $\CC$ is $\infty$-deletion-flat  \\
    $\Leftrightarrow$ $\CC$ has bd.\ $\infty$-splitter-rank 
    \smallskip\\
    \cite{flip-breakability} 
    \cite[by definition]{treedepth}
}

\newcommand{\shrubdepth}{
    $\CC$ has bounded shrub-depth\\
    $\Leftrightarrow$ $\CC$ is $\infty$-flip-flat  \\
    $\Leftrightarrow$ $\CC$ has bd.\ $\infty$-flipper-rank 
    \smallskip\\
    \cite{flip-breakability} 
    \cite[by SC-depth]{shrubdepth}
}

\newcommand{\treewidth}{
    $\CC$ has bounded tree-width\\
    $\Leftrightarrow$ $\CC$ is $\infty$-deletion-breakable  \\
    $\Leftrightarrow$ $\CC$ is $\infty$-deletion-separable
    \smallskip\\
    \cite{flip-breakability}
    \cite{graphminors2}
}

\newcommand{\cliquewidth}{
    $\CC$ has bounded clique-width\\
    $\Leftrightarrow$ $\CC$ is $\infty$-flip-breakable  \\
    $\Leftrightarrow$ $\CC$ is $\infty$-flip-separable
    \smallskip\\
    \cite{flip-breakability}
    [folklore, see \Cref{lem:infty-separability}]
}

\newcommand{\nowheredenseStable}{
    $\CC$ is nowhere dense\\
    $\Leftrightarrow$ $\CC$ is $r$-deletion-flat $\forall r \in \NN$ \\
    $\Leftrightarrow$ $\CC$ has bd.\ $r$-splitter-rank $\forall r \in \NN$ \smallskip\\
    \cite{dawar2010homomorphism,nevsetvril2011nowhere} 
    \cite{GKS17}
}

\newcommand{\monadicallystable}{
    $\CC$ is monadically stable\\
    $\Leftrightarrow$ $\CC$ is $r$-flip-flat $\forall r \in \NN$ \\
    $\Leftrightarrow$ $\CC$ has bd.\ $r$-flipper-rank $\forall r \in \NN$
    \smallskip\\
    \cite{flip-flatness}
    \cite{flipper-game}
}

\newcommand{\nowheredenseDependent}{
    $\CC$ is nowhere dense\\
    $\Leftrightarrow$ $\CC$ is $r$-deletion-breakable $\forall r \in \NN$ \\
    $\Leftrightarrow$ $\CC$ is $r$-deletion-separable $\forall r \in \NN$
    \smallskip\\
    \cite{flip-breakability} \cite{structural-sparsity}
}

\newcommand{\monadicallydependent}{
    $\CC$ is monadically dependent\\
    $\Leftrightarrow$ $\CC$ is $r$-flip-breakable $\forall r \in \NN$  \\
    $\Leftrightarrow$ $\CC$ is $r$-flip-separable $\forall r \in \NN$  
    \smallskip\\
    \cite{flip-breakability} 
    \cite{flip-separability}
}


 \NiceMatrixOptions{cell-space-limits = 5pt}
\newcolumntype{T}[1]{>{\raggedright\arraybackslash\setlength{\parskip}{0pt}\vspace{0pt}}p{#1}}

\begin{table*}[htbp]
   \centering
   \begin{NiceTabular}{|c|c|c|}[hvlines]
      & \textbf{sparse notions} & \textbf{dense notions} \\ 
      \Block{2-1}{\textbf{local} \\ \textbf{width} \\ \textbf{measures}} & \Block[l]{}{\nowheredenseDependent} & \Block[l]{}{\monadicallydependent}  \\
      & \Block[l]{}{\nowheredenseStable} & \Block[l]{}{\monadicallystable}  \\
      \Block{2-1}{\textbf{global} \\ \textbf{width} \\ \textbf{measures}}& \Block[l]{}{\treewidth} & \Block[l]{}{\cliquewidth}  \\
      & \Block[l]{}{\treedepth} & \Block[l]{}{\shrubdepth}  \\
   \end{NiceTabular}
   \medskip
   \caption{Characterizations of sparse graph classes based on vertex deletions and of dense graph classes based on flips.\label{table}}
 \end{table*}

Let us define the notions appearing in the \Cref{table} and explain how to apply \Cref{thm:engine} to them. To give uniform definitions we define the following sparse analog of a $k$-flip.

\begin{definition}
    We call a graph $H$ a \emph{$k$-deletion} of a graph $G$ if $H$ is obtained from $G$ by deleting at most $k$ vertices.    
\end{definition}

The first property we define is the \emph{flip-flatness} property from \Cref{ex:flip-flatness} in the introduction. 

\newcommand{\footnoteUQW}{The property of being $r$-deletion-flat for every $r\in\N$ was originally named \emph{uniform quasi-wideness}~\cite{dawar2010homomorphism,nevsetvril2011nowhere}. We renamed it here to provide a unified presentation.}

\begin{definition}
    For $r\in \NN \cup \{ \infty \}$, a graph class $\CC$ is \emph{$r$-flip-flat}, if there exist a \emph{margin} function $M\colon\NN \to \NN$ and a \emph{budget} $k \in \NN$ such that for all $m\in \NN$,  $G \in \CC$ and set $W \subseteq V(G)$ of size at least $M(m)$, there are a $k$-flip $H$ of $G$ and a size $m$ set $A \subseteq W$ such that for all distinct $u,v \in A$
    $$
        \dist_H(u,v) > r.
    $$
The definition of \emph{$r$-deletion-flatness}\footnote{\footnoteUQW} is derived by replacing $k$-flip with $k$-deletion.
\end{definition}

The next \emph{flip-breakability} property generalizes flip-flatness by relaxing the result condition. Instead of separating the vertices of a large set pairwise, we only demand two large sets that are separated from each other.

 \begin{definition}
    For $r\in \NN \cup \{ \infty \}$, a graph class $\CC$ is \emph{$r$-flip-breakable}, if there exist a \emph{margin} function $M\colon \NN \to \NN$ and a \emph{budget} $k \in \NN$ such that for all $m\in \NN$,  $G \in \CC$ and set $W \subseteq V(G)$ of size at least $M(m)$, there are a $k$-flip $H$ of $G$ and two disjoint size $m$ sets $A,B \subseteq W$ such that for all $u \in A$ and $v \in B$ we have
    $$
        \dist_H(u,v) > r.
    $$
    The definition of \emph{$r$-deletion-breakability} is derived by replacing $k$-flip with $k$-deletion.
 \end{definition}

The next \emph{flip-separability} property claims the existence of small balanced (flip-)separators.

 \newcommand{\footnoteSeparability}{
    The definitions of deletion- and flip-separability used in \cite{structural-sparsity,flip-separability} allow one to shrink the size of the neighborhoods to
    $\left|N_r^H[v] \cap W\right| \leq \left \lceil \eps \cdot \left|W\right| \right \rceil$
    for any fixed $\eps > 0$, at the cost of $k$ being also dependent on $\eps$.
    Additionally, \cite{flip-separability} uses weighted graphs.
    These more elaborate separability variants imply the version we present here. Moreover, in the two papers the proofs showing that deletion-separability (flip-separability) implies nowhere denseness (monadic dependence) only requires separability for $\eps = \tfrac12$ and unit weights, which makes the simplified definition we present here equivalent to the ones presented in \cite{structural-sparsity,flip-separability}.
 }

 \begin{definition}
    For $r\in \NN \cup \{ \infty \}$, a graph class $\CC$ is \emph{$r$-flip-separable}\footnote{\footnoteSeparability}, if there exists a \emph{budget} $k \in \NN$ such that for every $G \in \CC$ and set $W \subseteq V(G)$, there is a $k$-flip $H$ of $G$ such that for every $v \in V(G)$ we have
    $$
        \left|B^r_H(v) \cap W\right| \leq \left \lceil \tfrac12 \cdot \left|W\right| \right \rceil.
    $$
    The definition of \emph{$r$-deletion-separability} is derived by replacing $k$-flip with $k$-deletion.
 \end{definition}


We next define the \emph{flipper-} and \emph{splitter-rank}.

 \begin{definition}
    Fix $r\in \NN \cup \{\infty\}$ and $k\in\NN$. For the single vertex graph $K_1$ 
    we define 
    $$
        \frk_{r,k}(K_1) \coloneq  1.
    $$
    For every other graph $G$ we define
    $$
        \frk_{r,k}(G) \coloneq  1+ \min_{\substack{\text{$H$ is a} \\ \text{$k$-flip of $G$}}} \max_{v \in V(H)} \frk_{r,k}(H[B^r_H(v)]), \text{and}
    $$
    The definition of $\srk_{r,k}(G)$ is derived by replacing $k$-flip with $k$-deletion.
    A graph class $\CC$ has \emph{bounded $r$-flipper-rank} if there are $k,\ell \in \NN$ such that $\frk_{r,k}(G) \leq \ell$ for all $G \in \CC$.
    The definition of \emph{bounded $r$-splitter-rank} is derived by replacing $\frk$ with $\srk$.
\end{definition}

The flipper- and splitter-rank play a crucial role in the first-order model checking algorithms for monadically stable and nowhere dense graph classes~\cite{ssmc,dreier2024stablemc,GKS17}.
Intuitively, we can understand the flipper-rank (and also the splitter-rank) as a game played on a graph $G$ between two players \emph{flipper} and \emph{localizer}.
In every round of the game, flipper applies a $k$-flip to the current graph (with the goal of reducing it to a single vertex as fast as possible), 
and localizer shrinks the arena to an $r$-neighborhood (with the goal of surviving as long as possible).
The flipper-rank then denotes the number of rounds needed for flipper to finish the game.


\bigskip

We will now use \Cref{thm:engine} to prove the following theorem.

\begin{theorem}\label{thm:sparse-flatness}
    For every weakly sparse graph class $\CC$ and for every $r\in \NN \cup \{ \infty \}$ the following holds:
    \begin{enumerate}[label=(\Alph*)]
        \item If $\CC$ is $3r$-flip-breakable, then $\CC$ is also $r$-deletion-breakable.
        \item If $\CC$ is $3r$-flip-separable, then $\CC$ is also $r$-deletion-separable.
        \item If $\CC$ is $3r$-flip-flat, then $\CC$ is also $r$-deletion-flat.
        \item If $\CC$ has bounded $3r$-flipper-rank, then $\CC$ also has bounded $r$-splitter-rank.
    \end{enumerate}
\end{theorem}

Together with the characterizations in \Cref{table}, 
\begin{itemize}
    \item each of (A) and (B) individually implies both \eqref{it:cw} and \eqref{it:mon-dep} of \Cref{thm:sparse-classes}, and
    \item each of (C) and (D) individually implies both \eqref{it:sd} and \eqref{it:mon-stab} of \Cref{thm:sparse-classes}.
\end{itemize}
We start with the first three implications.

 \begin{proof}[{Proof of \Cref{thm:sparse-flatness}, items A, B, C}]
    We prove (A), while (B) and (C) are proved analogously.
    
        Assume that $\CC$ is $K_{t,t}$-free. Let $M\colon \NN \rightarrow \NN$ and $k\in \NN$ be the margin and budget for which $\CC$ is $3r$-flip-breakable.
        We claim that $\CC$ is $r$-deletion-breakable for margin $M$ and budget $k' \coloneq  kt^2$.
        Given a size $M(m)$ vertex subset $W$ in a graph $G\in\CC$, we can apply $3r$-flip-breakability which yields two disjoint size $m$ subsets $A$ and $B$ of $W$ and a $k$-flip $H$ in which all elements from $A$ are at distance greater than $3r$ from all elements in $B$.
        Using \Cref{thm:engine}, we obtain from $H$ a $kt^2$-deletion $G'$ of $G$ in which all elements from $A$ are at distance greater than $r$ from all elements in $B$.
 \end{proof}

The last item of \Cref{thm:sparse-flatness} is implied by the following lemma.

\begin{lemma}\label{lem:frk}
    For every $K_{t,t}$-free graph $G$ we have
    $$
        \frk_{3r,k}(G) \leq \ell
        \Rightarrow
        \srk_{r,k'}(G) \leq \ell,
    $$
    where $k' \coloneq  k^{(2^\ell)}t^2$.
\end{lemma}

In order to prove it, we gather some facts about the flipper-rank.

\begin{lemma}\label{lem:flipper-hereditary}
    For every graph $G$ and induced subgraph $H$ of $G$ we have
    $$
        \frk_{r,k}(H) \leq \frk_{r,k}(G).
    $$
\end{lemma}
\begin{proof}
    By induction on the flipper-rank using hereditariness (\Cref{lem:flip-basics}).
\end{proof}

\begin{lemma}\label{lem:flipper-transitive}
    For every graph $G$ and $k'$-flip $G'$ of $G$ we have
    $$
        \frk_{r,k\cdot k'}(G') \leq \frk_{r,k}(G).
    $$
\end{lemma}
\begin{proof}
    By symmetry (\Cref{lem:flip-basics}), $G$ is a $k'$-flip of $G'$.
    Let $H$ be the $k$-flip of $G$ witnessing a bound on $\frk_{r,k}(G)$.
    By transitivity (\Cref{lem:flip-basics}), $H$ is also a $(k \cdot k')$-flip of $G'$ that witnesses the same bound on $\frk_{r,k\cdot k'}(G')$.
\end{proof}

\begin{proof}[Proof of \Cref{lem:frk}]
    We show the lemma by induction on $\ell$. It holds for $\ell =1$, since in this case the graph $G $ is $K_1$.
    Assume we are given $G$ with $\frk_{3r,k}(G) \leq \ell + 1$.
    Let $H$ be the $k$-flip of $G$ witnessing this, i.e.,
    \begin{equation}\label{eq:srk-by-frk}
        \max_{v \in V(H)} \frk_{3r,k}(H[B^{3r}_H(v)]) \leq \ell.
    \end{equation}
    By \Cref{thm:engine} there exists a $kt^2$-deletion $G'$ of $G$ such that for every $v \in V(G')$
    $$
        B^r_{G'}(v) \subseteq B^{3r}_H(v).
    $$
    It follows from \Cref{lem:flipper-hereditary} and \eqref{eq:srk-by-frk} that
    \begin{equation*}
    	\max_{v \in V(G')} \frk_{3r,k}(H[B^{r}_{G'}(v)]) \leq \ell.
    \end{equation*}
    Furthermore, as $G'[B^{r}_{G'}(v)] = G[B^{r}_{G'}(v)]$ is a $k$-flip of $H[B^{r}_{G'}(v)]$ by symmetry (\Cref{lem:flip-basics}), it follows by \Cref{lem:flipper-transitive} that
        \begin{equation*}
    	\max_{v \in V(G')} \frk_{3r,k^2}(G'[B^{r}_{G'}(v)]) \leq \ell.
    \end{equation*}    
    Moreover, $G'$ is still $K_{t,t}$-free, so we can apply induction and conclude that
    $$
        \max_{v \in V(G')} \srk_{r,k'}(G'[B^{r}_{G'}(v)]) \leq \ell,
    $$
    where
    $$
        k' \coloneq   (k^2)^{(2^\ell)} t^2 = k^{(2^{\ell + 1})}t^2.
    $$
    Since $k' \geq kt^2$, $G'$ is a witness for $\srk_{r,k'}(G) \leq \ell + 1$, as desired.    
\end{proof}

\subsection{Strong flip-flatness}
 A graph class $\CC$ is \emph{strongly flip-flat} \cite{eleftheriadis2025extension} if for every $r$ it is $r$-flip-flat with a budget $k$ independent of~$r$.
 Similarly, it is \emph{uniformly almost-wide} if for every $r$ it is $r$-deletion-flat with a budget $k$ independent of $r$.
 If $\CC$ is a strongly flip-flat class of $K_{t,t}$-free graphs,
 with budget $k$ for all $r$,
 then by the same proof as for implication (C) of \Cref{thm:sparse-flatness}, 
 it is also uniformly almost wide with budget $k'=kt^2$ for all $r$.
 Hence, classes that are weakly sparse and strongly flip-flat are uniformly almost wide, reproving the main result of~\cite{ghasemi2025weaklysparsestronglyflipflatclasses}, stated in 
\Cref{thm:sparse-classes} \eqref{it:stronglyflip}.


\subsection{Flip-width}\label{sec:fw}

We now prove that weakly sparse graph classes of bounded flip-width have bounded expansion, thus proving \Cref{thm:sparse-classes} \eqref{it:fw}. This reproves a result of \cite{Tor23}, with a more direct proof and with better bounds.
Classes of bounded flip-width were introduced by Toruńczyk 
as a dense analogue of bounded expansion  \cite{Tor23}.

\begin{restatable}[{\cite[Thm. 6.3]{Tor23}}]{theorem}{thmFW}\label{thm:fw-sparsify}
    Every weakly sparse class of graphs of bounded flip-width has bounded expansion.
\end{restatable}

We will prove this through a variant of the Cops and Robber game in which the robber has bounded speed, considered by Toruńczyk \cite{Tor23}, which characterizes bounded expansion.

\paragraph*{Copwidth}
The \emph{radius-$r$ copwidth} of a graph $G$, denoted $\copw_r(G)$, is the least number of cops needed to catch a robber in the variation 
of the classic Cops and Robber where the robber has speed $r$.

More precisely,
the \emph{Cops and Robber game} of radius  $r\in\N\cup\set{\infty}$ and width  $k\in\N$, $k\ge 1$,
 proceeds in rounds, as follows.
Initially, the robber picks a vertex $w_0\in V(G)$, and the set of positions of the cops is $C_0\coloneqq \emptyset$.
In each round $i=1,2,\ldots$,
first the cops announce their set $C_i$ of new positions, with $|C_i|\le k$. Next, the robber moves from $w_{i-1}$ along any path of length at most $r$ avoiding the vertices $C_i\cap C_{i-1}$ occupied by the cops  which are not moving, ending in their next position $w_i$,
and the cops are now placed on the vertices of $C_i$. If $w_i\in C_i$, the cops win; otherwise, the game proceeds to the next round.
The \emph{radius-$r$ copwidth} of $G$ 
is the least number $k$ 
such that the cops have a winning strategy in the copwidth game of radius $r$ and width $k$.

\paragraph*{Flip-width} 
The \emph{radius-$r$ flip-width} is defined analogously to the Cops and Robber game, but instead of cops occupying vertices, they perform a $k$-flip of the graph.
 
The \emph{flipper game} of radius $r\in\N\cup\set{\infty}$ 
and width $k\in\N$, $k\ge 1$, is played on a graph $G$ by two players, \emph{flipper} and \emph{runner}.
In each round $i$ of the game,  
 a $k$-flip $G_i$ of $G$ is  declared by flipper, and the new position $v_i\in V(G)$ is selected by runner, as follows. Initially, $G_0=G$ and $v_0$ is a vertex of $G$ chosen by runner. In round $i>0$,
flipper announces a new $k$-flip $G_i$ of $G$, which will be put into effect momentarily.
Runner, knowing $G_i$, moves to a new vertex $v_i$ by following a path of length at most $r$ from $v_{i-1}$ to $v_i$ in the \emph{previous} graph $G_{i-1}$. The game terminates when the runner is trapped, that is, when $v_i$ is isolated in $G_{i}$. If the game terminates, flipper wins; else runner wins.

  The \emph{radius-$r$ flip-width} of a graph $G$,
  denoted $\fw_r(G)$,
  is the smallest number $k\in\N$ such that flipper has a winning strategy in the flipper game of radius $r$ and width $k$ on $G$.
  A class $\CC$ of graphs has \emph{bounded flip-width}
  if $\fw_r(\CC)<\infty$ for every $r\in\N$.
  More explicitly: for every radius $r\in\N$ there is some~$c_r\in\N$ such that $\fw_r(G)<c_r$ for all $G\in \CC$.

\medskip

It is shown in \cite[Cor 4.6]{Tor23}
that graph classes with $\copw_r(\CC)<\infty$ for all $r$ have bounded expansion. Thus, the following lemma implies that every weakly sparse class of bounded flip-width has bounded expansion, proving \Cref{thm:fw-sparsify}.

\begin{lemma}\label{lem:fw-cw2}
    For any $r \in \N \cup \{\infty\}$ and for any $K_{t,t}$-free graph $G$ we have

    \[
\copw_{r}(G)\le 2\cdot \fw_{3r}(G)\cdot t^2.   
\] 
\end{lemma}
\begin{proof}
  Let $G$ be a $K_{t,t}$-free graph such that $\fw_{3r}(G) = k$. 
  We will define a strategy for the cops following the flipper strategy. 
  Let $w_0\in V(G)$ be the initial position of the runner, set $G_0 = G$ and  $C_0 = S_0 = \emptyset$; these are the initial flip of the flipper, and the initial position of the cops. 
  
  For a round $i\geq 1$, let $G_i$ be the $k$-flip with partition $\P_i$ defined by the flipper. 
  We define the corresponding position $C_i$ of the cops:
  \[S_i = \Sparsify(G, \P_i, t) \quad \text{ and } \quad  C_i = S_i \cup S_{i-1}. \] 
  The position $w_i$ of the robber is at distance at most $r$ from its position $w_{i-1}$ in the graph 
  \[G\setminus (C_{i-1} \cap C_i) = G\setminus S_{i-1}.\]
  Clearly, by \Cref{lem:sparsify}, we have that 
  \begin{equation}\label{eq:cop_flipper}
    w_i \in B^r_{G\setminus S_{i-1}}(w_{i-1})  \subseteq B^{3r}_{G_{i-1}}(w_{i-1}).   
  \end{equation}  
  Since $\fw_{3r}(G) = k$, in the flipper game of radius $3r$ there is a winning strategy for the flipper using $k$-flips. Hence, in the Cops and Robber game of radius $r$, by \Cref{lem:sparsify-size,eq:cop_flipper}, there is a winning strategy for the cops with $2kt^2$ cops.

\end{proof}

We note that this simple proof cannot be extended to the blind versions of copwidth and flip-width (in which the evader is invisible to the pursuers). Indeed, there are weakly sparse classes that have bounded blind flip-width while having unbounded blind copwidth~\cite[Proposition~5.6]{blind_copwidth}. This is due to the fact that the winning conditions in those two games differ subtly -- in the blind flip-width game, the evader loses whenever his position is isolated, while in the blind copwidth game, he loses whenever he is captured by one of the cops. The above proof does lift to the blind setting, however, if the blind copwidth game is modified so that the evader loses whenever his position is isolated.

\section{Separation-width}\label{sec:sw}

In this section, we introduce the \emph{separation-width} parameters,
and argue that they occupy a sweet spot between previously studied graph parameters.
We start with recalling the \emph{degeneracy}, as well as the \emph{strong} and \emph{weak coloring numbers}, introduced by Kierstead and Yang \cite{generalizedcoloring}. We refer to \cite{SIEBERTZ2026100855} for a survey of those parameters, and their applications in algorithmic graph theory.

The degeneracy of a graph \(G\), denoted $\degeneracy(G)$, is the least number $k$ such that there is a total order $\le$ on $V(G)$ such that for every $v\in V(G)$, there is at most $k$ vertices $u\in N(u)$ with $u\leq v$.

\paragraph{Strong and weak coloring numbers} 
Fix a graph $G$,  $r\in\N\cup\set{\infty}$, and total order $\le$ on $V(G)$. 
A vertex \(u\) is \emph{strongly \(r\)-reachable} from a vertex \(v\) if $u\leq v$ and there is a path of length at most \(r\) from \(v\) 
to \(u\) such that all inner vertices on that path are larger than~\(v\).
 \emph{Weak \(r\)-reachability} is defined in the same way, except that all inner vertices are only required to be larger than~\(u\).
Let \(\sreach_r^\le(v)\) be the set of vertices that are strongly \(r\)-reachable from \(v\) with respect to the ordering $\le$, and \(\wreach_r^\le(v)\) be the set of vertices that are weakly \(r\)-reachable from \(v\).
The \emph{strong \(r\)-coloring number} of \(G\), denoted $\scol_r(G)$, is the least number $k$ such that there is some 
total order $\le$ on $V(G)$ 
with $\left|\sreach_r^\le(v)\right|\le k$ for all $v\in V(G)$.
The \emph{weak \(r\)-coloring number} of \(G\), denoted $\wcol_r(G)$, is defined analogously, by replacing strong reachability with weak reachability.
It is known that for all $r\in\N$ and every graph $G$,
\begin{align}\label{eq:scol_wcol}
\scol_r(G) &\leq \wcol_r(G) \leq \scol_r(G)^r,\\
\scol_1(G) &= \degeneracy(G) + 1 = \wcol_1(G),\label{eq:scol-deg}\\
 \scol_\infty(G)&=\tw(G)+1.\label{eq:scol-tw}
\end{align}


\begin{lemma}[\cite{Zhu09}]\label{fact:be}
	A graph class $\C$ has bounded expansion if and only if~$\scol_r(\C)<\infty$, for all~${r\in\N}$, equivalently, if and only if $\wcol_r(\C)<\infty$, for all~${r\in\N}$.
\end{lemma}




\paragraph{Separation-width}
We introduce a new family of graph parameters, dubbed \emph{separation-width},
and defined as follows.

\begin{definition}Fix $r\in\N\cup\set{\infty}$.
    Let $G$ be a graph, $S\subseteq V(G)$ a set of its vertices, and $v\in V(G)\setminus S$.
    The \emph{$r$-separator of $v$ in $S$}, denoted $\sep_r(v/S)$, is the set of those vertices in $S$ that can be reached from $v$ by a path of length at most~$r$, whose inner vertices are disjoint from $S$
    (this is called the $r$-projection of $v$ onto $S$ in \cite{kernelization}).
    
    The  \emph{radius-$r$ separation-width} of a graph $G$, denoted by $\smw_r(G)$,
is the least number $k$ such that there is a total order ${\le}$ on $V(G)$ such that $\left|\sep_r(v/S)\right|\le k$ for every $\le$-downward closed set $S\subseteq V(G)$ and for all $v\in V(G)\setminus S$.
\end{definition}

\subsection{Separation-width and generalized coloring numbers}
The following inequalities relate the  parameters defined above.
\lemScolSmw*
We remark that the 
second inequality in \eqref{eq:sandwich2} is a strengthening of an inequality implicit in \cite[Proof of Thm. 3]{dmtcs:14976}. 
Here, we provide a shorter, self-contained proof for completeness.
\begin{proof}
    Fix a total order $\le$ on $V(G)$.
    Let $v\in V(G)$, and let $S_{{<}v}$ be the set of vertices $w\in V(G)$ smaller than $v$ in the order.
    Observe that $$\sreach_r^{\le}(v)\ \subseteq\  \sep_r(v/S_{{<}v})\cup\{v\},$$
    which yields the first inequality in \eqref{eq:sandwich1}.
    On the other hand, for every $\le$-downward closed set $S\subseteq V(G)$
    and every $v\in V(G)\setminus S$,
    we have $$\sep_r(v/S)\cup\{v\}\ \subseteq\ \wreach_r^{\le}(v),$$
    which yields the second inequality in \eqref{eq:sandwich1}.
It remains to prove the second inequality in 
\eqref{eq:sandwich2}.
    For each $u\in \sep_r(v/S)$, let $\pi_u$ be a path of length at most $r$ from $v$ to $u$, whose inner vertices are disjoint from $S$.
    Let $m_u$ be the $\le$-smallest vertex of $V(\pi_u)-S = V(\pi_u) - \{ u \}$
    (such a vertex exists since $v\notin S$).
    Let $w$ be the $\le$-smallest among $\{m_u: u \in \sep_r(v/S)\}$. 
    In particular, there is a path $\pi$  from $w$ to $v$
    with inner vertices ${\ge} w$, and of length at most $r-1$.
    Then for 
    every $u\in \sep_r(v/S)$,
    the concatenation of $\pi$ and $\pi_u$ is a path from $w$ to $u$, whose inner vertices are ${\ge} w$, and of length at most $2r-1$.
    Then, by possibly shortcutting paths that visit $w$ multiple times, we can assume that all inner vertices are ${>} w$.
    It follows that  $$\sep_r(v/S)\cup\set{w}\subseteq \sreach_{2r-1}^{\le}(w),$$
    proving the second inequality in 
\eqref{eq:sandwich2}.
\end{proof}

\Cref{lem:scol_smw}, together with \eqref{eq:scol-deg} and \eqref{eq:scol-tw}, yields the following.
\corSMWDEGTW*
Furthermore, \Cref{lem:scol_smw}, together with \Cref{fact:be}, yield the following.
\corSMWBE*
Together with known results, \Cref{lem:scol_smw} shows that $\smw_r$ can be significantly smaller than $\wcol_r$.
For instance, for every $r,t\in \N$ there is a graph $G$ of tree-width $t$ with  $\wcol_r(G)= \binom{r+t}{t}=\Theta_r(t^r)$ \cite[Sec. 4]{DBLP:conf/wg/GroheKRSS15},
while by \Cref{cor:smw-tw}, $\smw_r(G)\le \smw_\infty(G)= t$.

\subsection{Separation-width and copwidth}\label{sec:copw}




We observe that separation-width can be used to provide monotone winning strategies in the Cops and Robber game (see \Cref{sec:fw} for its definition).  
The following result is implicit in the work
of Hickingbotham \cite{dmtcs:14976}.
Again, we provide a shorter, self-contained proof for completeness.

\begin{restatable}{theorem}{thmCopw}\label{lem:copw}For every graph $G$ and $r\in\N$,
$$\copw_r(G)\le \smw_{2r}(G)+1.$$    
Moreover, there is an enumeration $v_1,v_2,\ldots,v_n$ of the vertices of $G$ 
such that $\smw_{2r}(G)+1$ cops have a winning strategy where in round $i$,
the cops are constrained to vertices 
$v_1,\ldots,v_i$, and the robber is constrained to vertices $v_{i+1},\ldots,v_n$. In particular, this strategy is monotone.
\end{restatable}

Here \emph{monotone} is in the sense that the robber never visits a vertex which was previously occupied by a cop.

\begin{proof}
Order the vertices $v_1<v_2<\ldots<v_n$ according to the order witnessing $\smw_{2r}(G)$. Let $S_i=\set{v_1,\ldots,v_{i}}$
for $i\in[n]$ and $S_0=\emptyset$.

Let $w_0\in V(G)$ be the initial position of the robber, and set $C_0=\emptyset$.
In round $i$, for $i=1,2,\ldots$, the cops are directed to each vertex in $\sep_{2r}(w_{i-1}/S_{i-1})$, and also to $v_{i}$;
 that is, $$C_i\coloneqq \sep_{2r}(w_{i-1}/S_{i-1})\cup\set{v_{i}}.$$
Next, the robber moves to their new position $w_i$
along a path $\pi_{i}$ of length at most $r$ from the previous position $w_{i-1}$, avoiding  $C_{i-1}\cap C_{i}$, and loses the game if $w_i\in C_i$.

We prove that as long as the game proceeds,  the following invariant is maintained:
\begin{equation}\label{eq:inv}
\textit{The path $\pi_i$  is disjoint from $S_{i-1}$.}
\end{equation}
This clearly holds for $i=1$ as $S_0=\emptyset$.
Suppose \eqref{eq:inv} holds in round $i\ge 1$.
In particular $w_i\notin S_{i-1}$.
If $w_i\in S_i = S_{i-1} \cup \{ v_i \}$ then $w_i=v_i\in C_i$ and  the robber is caught in round $i$, so assume that $w_i\notin S_i$.
We prove that $\pi_{i+1}$ is disjoint from $S_i$.

Otherwise,
let $s$ be the first vertex in $S_i$ visited by $\pi_{i+1}$; then $s\in \sep_{r}(w_i/S_i)\subseteq C_{i+1}$. As $\pi_{i+1}$ avoids $C_i\cap C_{i+1}$, we have $s\notin C_i$.
In particular, $s\neq v_i$, so $s\in S_{i}\setminus \set{v_i}=S_{i-1}$. 
 The concatenation of the paths $\pi_{i}$ and $\pi_{i+1}$ is a path of length at most $2r$, and $s$ is the first vertex in $S_{i-1}$ on this path by \eqref{eq:inv}.
 Hence, $s\in \sep_{2r}(w_{i-1}/S_{i-1})\subseteq C_{i}$, a contradiction. Hence, $\pi_{i+1}$ is disjoint from $S_i$, maintaining the invariant.

By \eqref{eq:inv}, for $i\in 1,2,\ldots$, either the robber loses within the first $i$ rounds, or 
$w_i\in V(G)\setminus S_i=\set{v_{i+1},\ldots,v_n}$. In particular, the robber must lose within $n$ rounds.
\end{proof}


\section{Merge-width in sparse graphs}\label{sec:mw-smw}

In this section, we reprove the following result of Dreier and Toruń\-czyk.

\thmMW*

As opposed to \cite{DT25}, our proof is direct and provides explicit bounds, as stated in the upcoming \Cref{lem:bounded_smw}. We first recall the definition of merge-width.

\subsection{Merge-width}
We refer the reader to \cite[Sec. 1]{DT25} for the original definition of the merge-width parameters via construction sequences,
and for examples illustrating this notion.
In this paper, it  will be more convenient to use the upcoming equivalent definition of merge-width through restrained flip sequences, following \cite[Def. 7.9 in arXiv version]{DT25}.

\begin{definition}Let $G$ be a graph with vertices $V$.
	A \emph{restrained flip sequence} for $G$ is a sequence
	\begin{align}\label{eq:mono-fw}
		(\P_1,R_1,G_1),\ldots,(\P_m,R_m,G_m)    
	\end{align}
	such that:
	\begin{itemize}
		\item $\P_1, \dots ,\P_m$ is a refining sequence of partitions of $V(G)$, starting with the trivial partition $\P_1 = \{V(G)\}$, ending with the partition $\P_m$ into singletons, where each partition $\P_i$ is obtained by splitting parts of $\P_{i-1}$ (possibly none),
		\item The graph $G_i$ is a $\P_{i}$-flip of $G$, for $i\in [m]$,
		\item ${V\choose 2}= R_1\supseteq \cdots \supseteq R_m=\emptyset$, and
		\item  $E(G_i)\subseteq R_i$ for $i\in[m]$. 
	\end{itemize}
\end{definition}

Fix $r\in\N$.
For a set $R\subseteq {V\choose 2}$ and partition $\P$ of $V$,
we say a part \(A \in \P\) is \emph{\(r\)-reachable} from a vertex \(v\) if there is a path of length at most $r$ in the graph $G=(V,R)$ from $v$ to some vertex in~$A$.
The \emph{radius-$r$ width} of $(\P,R)$  is 
\[
\max_{v \in V}\ \Big|\bigl\{ A \in \P \mid  A \text{ is \(r\)-reachable from } v \bigr\}\Big|.
\]

The \emph{radius-$r$ width} of the restrained flip sequence is 
the maximum, over $i\in[m-1]$,
of the radius-$r$ width of $(\P_{i+1},R_{i})$.

\begin{definition}
	The \emph{radius-$r$ merge-width} of a graph $G$, denoted $\mw_r(G)$, is the least radius-$r$ width of a restrained flip sequence of  $G$.
	A graph class $\C$ has \emph{bounded merge-width} if for all $r\in \N$ we have\footnote{For a graph parameter $f$ and graph class $\C$, by $f(\C)$ we denote $\sup_{G\in\C}f(G)$.}  
	${\mw_r(\C)<\infty}$.
\end{definition}

\begin{observation} \label{obs:proper_sequence}
	Given a graph $G$ and a restrained flip sequence of $G$ of radius-$r$ width $c_r$ for $r \in \mathbb{N}\cup \{\infty\}$, there is a restrained flip sequence $(\P_1,R_1,G_1),\ldots,(\P_n,R_n,G_n) $ of equal radius-$r$ width for every $r \in \mathbb{N}\cup \{\infty\}$, such that $n = \left|V(G)\right|$ and $\left|\P_{i+1}\right| = \left|\P_i\right| + 1$.
\end{observation}
\begin{proof}
	Let $(\P_1,R_1,G_1),\ldots,(\P_m,R_m,G_m)$  be a restrained flip sequence of $G$. First note that if $\P_i=\P_{i+1}$ for some $1\le i<m$, then we can 
	drop the triplet $(\P_i,R_i, G_i)$ in the restrained flip sequence, obtaining a sequence of equal radius-$r$ width  (this is because the pair $(\P_{i},R_{i-1})$ 
	is equal to the pair $(\P_{i+1},R_{i-1})$ and hence they have the same radius-$r$ merge-width).
	It follows that we may assume, without loss of generality, that the partitions $\P_1,\ldots,\P_m$ are pairwise distinct. 	
	Next, observe that if $\left|\P_{i+1}\right| > \left|\P_i\right| + 1$, then there is a partition $\P$ which is strictly finer than $\P_i$ and strictly coarser than $\P_{i+1}$. In this case we may insert the triplet $(\P,R_i, G_i)$ into the restrained flip sequence, without increasing its radius-$r$ merge-width.
\end{proof}

\subsection{Forward direction}

The forward implication of \Cref{thm:mw-sparsify} follows by \cite[Lem. 7.5]{DT25},
which we reformulate here in terms of separation-width.

\begin{lemma}[{\cite[Lem. 7.5]{DT25}}]\label{lem:smw-mw}
    Fix $r\in\N$. Let $G$ be a graph and $k= \smw_{r+1}(G)$. 
    Then $$\mw_r(G)\le 2^{k + 1} + 1.$$
\end{lemma}

Indeed, if $\CC$ has bounded expansion then $\CC$ is weakly sparse\footnote{observe that $\scol_2(K_{t,t})\geq t$.},
and $\smw_{r}(\CC)<\infty$ for all $r\in\N$ by \Cref{cor:smw-be}. Consequently, by the bound stated in \Cref{lem:smw-mw}, also $\mw_r(\CC)<\infty$ for all $r\in\N$, so $\CC$ has bounded merge-width. 
We provide now the proof of \Cref{lem:smw-mw} as it illustrates the tight connection between separation-width and merge-width.

\begin{proof} 
    Let $\le$ be a total order witnessing that $\smw_{r+1}(G)= k$.
    Order the vertices $v_1<v_2<\ldots<v_n$ according to this order.
    For $i=1,\ldots,n$, denote $v_{{<}i}\coloneqq\set{v_1,\ldots,v_{i-1}}$, and 
    let $G_i$ be the graph obtained from $G$ by isolating the vertices in $v_{{<}i}$.
    Let $\mathcal P_i$ be the partition of $V(G)$ which partitions $v_{{<}i}$ into singletons, while the remaining vertices in $\set{v_{i},\ldots,v_n}$ are partitioned according to their neighborhood in $v_{{<}i}$. Then $G_i$ is a $\mathcal P_i$-flip of $G$.
    Let $R_i=E(G_i)$.
    Then  $$(\mathcal P_1,R_1,G_1),\ldots, (\mathcal P_n,R_n,G_n)$$ forms a restrained flip sequence.
    It is not difficult to argue
     that 
    its radius-$r$ width is bounded by 
    $2^{k + 1} + 1$; a more precise bound is stated in \Cref{lem:mw-smw} below.
    Namely, for a fixed vertex $v\in V(G)$, if a part of $A\in \P_{i}$ is reachable from $v$ by a path of length at most $r$ in $G_{i}$, then 
     $$N_G(A)\cap v_{{<}i}\ \subseteq\ \sep_{r+1}(v/v_{{<}i}).$$ As $A\in \P_i$ is uniquely determined by $N_G(A)\cap v_{{<}i}$
     and $\left|\sep_{r+1}(v/v_{{<}i})\right|\le k$, it follows that there are at most $2^k$ many such parts $A$. Observe that for such a part $A$, if $v_i \not \in A$, there are at most $2$ parts of $\P_{i+1}$ contained in $A$: $A \cap N(v_i)$ and $A \setminus (N(v_i) \cup \{v_i\})$,
    while if $v_i \in A$, there are at most $3$ such parts (including $\{v_i\} \in \P_{i+1}$). Therefore, there are at most $2^{k+1} + 1$ parts of $\P_{i + 1}$ that are $r$-reachable from $v$ in $G_i$.
\end{proof}

We remark that the proof of \cite[Lemma 7.19]{DT25} actually shows the following bound, improving the bound in \Cref{lem:smw-mw}, while only stating the weaker bound in which $\smw$ is replaced by $\wcol$.
\begin{lemma}\label{lem:mw-smw}
    For every $r\in\N$ and graph $G$,
    we have $$\mw_r(G)\le 2 \cdot \pi_G(\smw_{r+1}(G)) + 1.$$ 
\end{lemma}
Here, $\pi_G$ is the \emph{shatter function} of $G$, defined in \cite{DT25} by $$\pi_G(m)\coloneqq \max_{A\subseteq V(G),|A|\le m}\Big(\left|A\right|+\left|\{N(v)\cap A: v\in V(G)\setminus A\}\right|\Big).$$
Clearly, for every graph $G$ we have that $\pi_G(m)\le m+2^m$, 
so \Cref{lem:mw-smw} implies in particular that $\mw_r(G)$ is bounded in terms of $\smw_{r+1}(G)$ in all graphs.
Furthermore, if $G$ is $K_{t,t}$-free then $\pi_G(m)\le \O(m^t)$, see e.g. \cite[Lemma B.3]{Tor23}.
Moreover, for classes $\C$ of bounded expansion, $\pi_G(m) \le \O_\C(m)$ for all $G\in\C$ \cite{REIDL2019152}.
We obtain the following (the constants hidden in the notation $\O_\CC$ depend on  $\CC$).

 \begin{restatable}{corollary}{corSMWMWBE}\label{cor:smw}
    For every graph class $\CC$ of bounded expansion, $r\ge 1$, and $G\in\CC$
    \[\mw_{r}(G) \le \O_\CC(\smw_{r+1}(G)).\]
\end{restatable}

\subsection{Backward direction}
The backward implication of \Cref{thm:mw-sparsify} follows by combining \Cref{cor:smw-be} with the following lemma, which is our main technical contribution.

\begin{restatable}{lemma}{lemSmw} \label{lem:bounded_smw}
    Fix $d,t,r\in\N$ with $r\ge 1$.
    Let $G$ be a $K_{t,t}$-free graph with $\mw_{3r+1}(G) \leq d$. Then 
    \[\smw_r(G) \leq \O(d^2t^3).\]
\end{restatable}
Note that \Cref{lem:bounded_smw}
implies \Cref{lem:bounded_scol}, by \cref{lem:scol_smw}.

\begin{proof}
    Let $(\P_1, R_1, G_1), \dots , (\P_n , R_n , G_n)$ be a restrained flip sequence of $G$ of radius-$(3r+1)$ width $d$,
    where $\P_1$ is a partition into one single part, $\P_n$ is a partition into singletons. We can assume without loss of generality by \cref{obs:proper_sequence} that $\left|\P_{i+1}\right| = \left|\P_i\right| + 1$.
    Define \[S_i \coloneq  \Sparsify(G,\P_i, t), \quad S_{{\le}i} \coloneq  \bigcup_{j\le i} S_j, \quad \Delta_i \coloneq  S_i\setminus \bigcup_{j<i} S_{j}.
    \]
    Observe that $S_n = V(G)$ as $\P_n$  is the partition into singletons, which are $t$-small parts.
    Note that $\left|\Delta_i\right| < 2t^2$ for all $i\in[n]$, by \Cref{lem:sparsify_refinement}.

    Pick any order on $V(G)$
    which first places the vertices of $\Delta_1=S_{1}$, followed  by the vertices of $\Delta_2$, etc., followed by the vertices of $\Delta_n$, where within each of those $n$ sets, the vertices are ordered arbitrarily.
    We will show that this order
    witnesses that 
    $\smw_r(G) \leq \O(d^2t^3)$.

    Fix a $\le$-downward closed set $S$ of $V(G)$ with respect to this order,
    and a vertex $v\in V(G)\setminus S$.
    We need to bound $\left|\sep_r(v/S)\right|$.
    Note that $S=S_{{\le}i}\cup S'$ for some $i\in\{0,\ldots,n-1\}$ and $S'\subseteq \Delta_{i+1}$, where $S_{{\le}0}=\emptyset$. Moreover,
     if $u\in \sep_r(v/S)$,
     then either $u\in \Delta_{i+1}$ (and there are less than $2t^2$ vertices in $\Delta_{i+1}$),
     or $u\in S_{{\le}i}$ and,
     because $S_{{\le}i}\subseteq S$,
     then $u\in \sep_r(v/S_{{\le}i})$.
     Hence,
     \begin{equation}\label{eq:mr}
     \left|\sep_r(v/S)\right| < 2t^2+\left|\sep_r(v/S_{{\le}i})\right|.   
     \end{equation}     
      Denote $X\coloneqq \sep_r(v/S_{{\le}i})$; it  therefore remains to bound $\left|X\right|$. If  $i=0$ then $X= S_{{\le}i}=\emptyset$, so assume $i>0$.


    By definition of $X$, each    
    $u \in X$ is adjacent to some  $w\in B^{r - 1}_{G \setminus  S_{{\le}i}}(v)$.
    Call such a vertex $w$ a \emph{witness} of $u$. Note that since $w\notin S_i$, the part of ${\cal P}_i$ containing $w$ is $t$-big. 
    Moreover, as $w \in B^{r-1}_{G\setminus S_{{\le}i}}(v)\subseteq
    B^{r-1}_{G\setminus S_{i}}(v)$, by \Cref{lem:sparsify},
    \begin{equation}\label{eq:close_witnesses}
        w \in B^{3r-3}_{G_i }(v).
    \end{equation}
    Hence, $w$ is in a $t$-big part which intersects $B^{3r-3}_{G_i}(v)$.
    There are at most $d$ such parts,
    by the assumption on the radius-${3r}$ merge-width of $G$.

    We start with bounding $\left|X \setminus B^{3r}_{G_i}(v)\right|$.
    Pick $u\in X\setminus B^{3r}_{G_i}(v)$ and a witness $w$ of $u$.
    Let $P_u, P_w \in \P_i$ be such that $u \in P_u$ and  $w \in P_w$.
    As observed above, the part $P_w$ is $t$-big, and $w \in B^{3r-3}_{G_i }(v)$.    
    As $u \not \in B^{3r}_{G_i}(v)$
    and $u$ is a neighbor of $w$ in $G$, it follows that $P_u$ and $P_w$ must be flipped in $G_i$.

    \begin{claim}
        If the part $P_u$ is $t$-big, then $u$ is $t$-complete to $P_w$.
    \end{claim}
    \begin{claimproof}
    Let 
    \[
        P_w' \coloneq  P_w \setminus N_{G}(u)
        \quad
        \text{and}
        \quad
        P_u' \coloneq  P_u \setminus N_{G}(w).
    \]
    By the same reasoning as in \Cref{clm:dist-to-set}, for all $x \in P'_w$ and $y \in P'_u$
    \begin{equation}\label{eq:short-distance-again}
        \dist_{G_i}(u,x) \leq 1
        \quad\text{and}\quad
        \dist_{G_i}(w,y) \leq 1.
    \end{equation}
    Then $P_u'$ and $P_w'$ are disjoint.
    Indeed, if there was some $x\in P_u'\cap P_w'$, then $u-x-w$ would form a path of length (at most) $2$ in $G_{i}$, by \eqref{eq:short-distance-again}.
    This, together with $w\in B^{3r-3}_{G_i}(v)$ would imply $u \in B^{3r}_{G_i}(v)$, contradicting the assumption.

    As $w \not \in S_i$, we also know that $w$ is not $t$-complete to the $t$-big part $P_u$, thus $\left|P_u'\right| \geq t$.
    Now suppose that $u$ is not $t$-complete to $P_w$, i.e., $\left|P_w'\right| \geq t$.
    If there is no edge between $P_u'$ and $P_w'$ in $G_i$, then $P_u'$ and $P_w'$ yield a $K_{t,t}$ in $G$, a contradiction.
    Similarly, if there is an edge $xy\in E(G_i)$ with $x\in P_u'$ and $y\in P_w'$, then $u-x-y-w$ form a path of length (at most) $3$ in $G_{i}$, by \eqref{eq:short-distance-again}.
    This again contradicts $u \notin B^{3r}_{G_i}(v)$.
    Hence, $u$ is $t$-complete to $P_w$.
\end{claimproof}

    Since there are at most $d$ many $t$-big parts in $\P_i$ that are $(3r-3)$-reachable from $v$ in $G_i$, and, by \Cref{lem:almost_complete}, for each such part $P$ there are less than $t$ vertices that are $t$-complete to $P$.
    We conclude:
    \begin{bound}\label{bound:1}
        There are less than $dt$ vertices of $X \setminus B^{3r}_{G_i}(v)$ that are in $t$-big parts of $\P_i$.
    \end{bound}

    \medskip

    Now consider the vertices of $X \setminus B^{3r}_{G_i}(v)$ that are in $t$-small parts.
    As in the previous case, each such vertex also must have a witness in $B^{3r-3}_{G_i}(v)$ which is in a $t$-big part.
    \begin{claim}
        Let $P_W \in \P_i$ be a $t$-big part intersecting $B^{3r-3}_{G_i}(v)$
    and $\P_U \subseteq \P_i$ be the set of $t$-small parts containing vertices of $X \setminus B^{3r}_{G_i}(v)$ that have a witness in $P_W$.
    Then $\left|\cal P_U\right|<dt+t$.
    \end{claim}
    \begin{claimproof}            
    Let $P_W$ and $\P_U$ be as in the statement of the claim.
    For every part $P_u \in \P_U$, there is a vertex $u \in P_u \cap (X \setminus B^{3r}_{G_i}(v))$, and a witness $w \in P_W$ of $u$. Hence, by $u \notin B^{3r}_{G_i}(v)$ and \eqref{eq:close_witnesses}, we must have that every $P_u \in \P_U$ is flipped with $P_W$.

    Since $P_W$ is $t$-big, it contains at least $t$ vertices $w_1, \ldots, w_t$ (which are not necessarily witnesses). For each $j\in [t]$, the vertex $w_j$ has non-neighbors in at most $d$ parts of $\P_U$ in $G$ (because in $G_i$, $w_j$ cannot reach more than $d$ parts by the bound on the merge-width).
    Then, if $\left|\P_U\right| \geq dt + t$, at least $t$ parts of $\P_U$ are complete to $w_1, \ldots, w_t$, yielding a $K_{t, t}$ in~$G$.
    Therefore, $\left|\P_U\right| < dt + t$. 
    \end{claimproof}

    Since there are at most $d$ many $t$-big parts $P_W \in \P_i$ which intersect $ B^{3r - 3}_{G_i}(v)$,
    and for each such $P_W$ there are less than $dt + t$ many $t$-small parts containing some vertex $u \in X \setminus B^{3r}_{G_i}(v)$ with a witness in $P_W$,
    we conclude:
    \begin{bound}\label{bound:2}
        There are less than $d (d+1)t^3$ vertices of $X \setminus B^{3r}_{G_i}(v)$ that are in $t$-small parts of $\P_i$.
    \end{bound}

    \medskip
    It remains to bound $\left|X \cap B^{3r}_{G_i}(v)\right|$.
    For a part $P \in \P_i$, denote
    \[
        U_P \coloneq X \cap B^{3r}_{G_i}(v) \cap P.
    \]
    By the bound on the merge-width, there can be at most $d$ parts $P \in \P_i$, for which $U_P$ is non-empty.
    Thus, we immediately get:
    \begin{bound}\label{bound:3}
        There are less than $dt^2$ vertices of $X \cap B^{3r}_{G_i}(v)$ that are in parts $P \in \P_i$ with $\left|U_P\right| < t^2$.    
    \end{bound}
    
    Now consider only the parts $P$ with $\left|U_P\right| \geq t^2$. These parts are $t$-big,
    so for each $u \in U_P$ the reason for $u$ belonging to $S_{{\le}i}$ must be that $u$ is $t$-complete to some $t$-big part $Q' \in \P_{j}$, where $j \leq i$.
    Let $i_u$ be the largest $j$ such that $u$ is $t$-complete to some $t$-big part $Q' \in \P_{j}$, and call such $Q'$ a \emph{latest reason} for $u$. For every $u\in U_P$ pick a latest reason $Q_u$ for $u$. 
    \begin{claim} \label{cla:reason_distinct_equal}
    	For all $u,w \in U_P$, either $Q_u=Q_w$ or $Q_u \cap Q_w = \emptyset$.
    \end{claim}    
    \begin{claimproof}
    	Suppose $Q_u \cap Q_w  \neq \emptyset$, and assume $Q_u \in \mathcal P_{i_u}$ and $Q_w \in \mathcal P_{i_w}$ with $i_u\leq i_w$. Since $\mathcal P_{i_w}$ is a refinement of $\mathcal P_{i_u}$, it follows that $Q_w \subseteq Q_u$. Note that $u$ is $t$-complete to every subset of $Q_u$, in particular to $Q_w$, hence by maximality of $i_u$, $Q_u = Q_w$.
    \end{claimproof}  
    See \cref{fig:mw_bound3} for an illustration of the proof of following claims.  

\begin{figure}[]
	\centering
	\includegraphics[scale = 1.6]{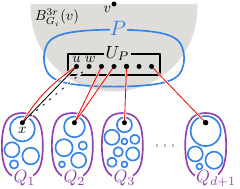}
	\caption{Illustration of the setting of \cref{bound:3} in  \cref{lem:bounded_smw}. Edges in black are from the original graph $G$, edges in \textcolor{red}{red} are from the $\P_i$-flip $G_i$. The partition \textcolor[rgb]{0.208, 0.518, 0.894}{$\P_i$} is in \textcolor[rgb]{0.208, 0.518, 0.894}{blue}, the set of parts \textcolor[rgb]{0.569, 0.255, 0.675}{$\Q$} is in \textcolor[rgb]{0.569, 0.255, 0.675}{purple}. Here $Q_u = Q_1$.}
	\label{fig:mw_bound3}
\end{figure}

    \begin{claim}\label{cla:reason_reachable}
    	For every $u\in U_P$, there is a part $Q \in \P_i$ with $Q\subseteq Q_u$ which is $(3r+1)$-reachable from $v$ in $G_i$.
    \end{claim}
     \begin{claimproof}
     	By \cref{lem:almost_complete},  less than $t$ vertices are $t$-complete to $Q_u$, so as $|U_P| \geq t^2$, there exists a vertex $w \in U_P$ with at least $t$ non-neighbors in $Q_u$.
     	On the other hand, $u$ has less than $t$ non-neighbors in $Q_u$. So there exists a vertex $x \in Q_u$ such that $wx \not \in E(G)$ and $ux \in E(G)$.
        Since $G_i$ is a $\mathcal P_i$-flip of $G$ and $u,w$ are in the same part $P$ of $\mathcal P_i$,
        it follows that  either $wx \in E(G_i)$ or $ux \in E(G_i)$. Since  $u,w\in B^{3r}_{G_i}(v)$, it follows that $x$ is $(3r+1)$-reachable from $v$ in $G_i$.
     \end{claimproof}

     \begin{claim}
        $|U_P| < dt$.
     \end{claim}
     \begin{claimproof}
        Consider the set $ \mathcal{Q}\coloneq \{Q_u : u \in U_P \}$ containing the latest reason of each vertex in $U_P$.
        By \Cref{lem:almost_complete}, each part in $\mathcal{Q}$ is the latest reason for fewer than $t$ vertices. Hence, it suffices to prove $|\mathcal{Q}|\leq d$.
        Assume towards a contradiction that $\cal Q$ contains distinct parts $Q_1, \ldots, Q_{d+1}$.
        By \Cref{cla:reason_reachable}, there are parts 
        \[
            Q'_1\subseteq Q_1, \ldots, Q'_{d+1} \subseteq Q_{d+1}
        \]
        from the partition $\P_i$ that are all reachable from $v$.
        Since $Q_1, \ldots, Q_{d+1}$ are distinct, they must be pairwise disjoint by \Cref{cla:reason_distinct_equal}.
        Thus, $Q'_1, \ldots, Q'_{d+1}$ are distinct parts of $\P_i$.
        As they are all reachable from~$v$, we get a contradiction to the merge-width bound.
     \end{claimproof}
    
    By the bound on the merge-width, there can be at most $d$ parts $P \in \P_i$, for which $U_P$ is non-empty. Thus, by applying the above claim, we conclude:

    \begin{bound}\label{bound:4}
        There are less than $d^2t$ vertices of $X \cap B^{3r}_{G_i}(v)$ that are in parts $P \in \P_i$ with $\left|U_P\right| \geq t^2$.    
    \end{bound}

    Adding up \Cref{bound:1,bound:2,bound:3,bound:4}, we can exhaustively bound
    $$\left|\sep_{r}(v/S_{{\le}i})\right|=\left|X\right| \le \O(d^2t^3).$$
    Together with \eqref{eq:mr},
    this proves $\left|\sep_{r}(v/S)\right|\le \O(d^2t^3)$ for any set $S$ that is downward closed with respect to the considered order and $v\in V(G)\setminus S$,
    and thus $\smw_r(G)\le\O(d^2t^3)$.
\end{proof}


\printbibliography

\appendix
\section{Flip-separability characterization of clique-width}

The following lemma is folklore. We provide a quick proof sketch for completeness.

\begin{lemma}\label{lem:infty-separability}
    A graph class $\CC$ has bounded clique-width if and only if it is $\infty$-flip-separable.
\end{lemma}

\begin{proof}[Proof sketch]
    One can show that every class of bounded clique-width is $\infty$-flip-separable by working with the rank-width characterization of clique-width, analogously to the proof of $\infty$-flip-breakability in the full version of \cite{flip-breakability}.
    There, the authors prove that every class of bounded clique-width is $\infty$-flip-breakable with margin function $M(m) = 4m$.
    One can also use this result as a black box and recursively break the largest \emph{unbroken} (i.e., contained in the same connected component) subset of $W$. After $3$ iterations, the largest unbroken subset will have size $\lceil (3/4)^3 \cdot |W| \rceil =  \lceil \frac{27}{64} \cdot |W| \rceil \leq \lceil 1/2 \cdot |W| \rceil$ and we are done.

    To show that $\infty$-flip-separability implies $\infty$-flip-breakability (and therefore bounded clique-width), notice that once we separated a set $W$, we can partition it into sets $A$ and $B$ with $|A| \geq |B| \geq |W|/4$ such that every vertex from $A$ is in a different connected component from every vertex in $B$:
    Order the connected components $C_1,\ldots, C_k$ of the flipped graph by how many vertices from $W$ they contain. There will be a prefix of the order that contains at least $1/4$ and at most $3/4$ of the vertices from $W$.
\end{proof}

\end{document}